\documentclass[bj,numbers,noshowframe]{imsart}

\usepackage[utf8]{inputenc}
\usepackage{graphicx}
\usepackage{amsmath,amssymb,amsfonts,amsthm}
\usepackage{algorithm,mathtools,bbm}
\usepackage{xcolor}
\usepackage{tikz}
\mathtoolsset{showonlyrefs}

\startlocaldefs
\newtheorem{remark}{Remark}
\newtheorem{lemma}{Lemma}
\newtheorem{theorem}{Theorem}

\newtheorem{definition}{Definition}
\newtheorem{corollary}{Corollary}

\endlocaldefs

\begin{document}
\begin{frontmatter}

\title{A semicircle law for the normalized Laplacian of sparse random graphs}
\runtitle{A semicircle law for the normalized Laplacian}

\begin{aug}
\author[A]{\inits{Y.}\fnms{Yiming}~\snm{Chen}\ead[label=e1]{ymchenmath@math.pku.edu.cn}}
\author[B]{\inits{Z.}\fnms{Zijun}~\snm{Chen}\ead[label=e2]{czj4096@gmail.com}}
\author[C]{\inits{Y.}\fnms{Yizhe}~\snm{Zhu}\ead[label=e3]{yizhezhu@usc.edu}}

\address[A]{School of Mathematical Sciences, Peking University\printead[presep={,\ }]{e1}}

\address[B]{School of Mathematics, Shandong University\printead[presep={,\ }]{e2}}

\address[C]{Department of Mathematics, University of Southern California\printead[presep={,\ }]{e3}}
\end{aug}

\begin{abstract}
We study the limiting spectral distribution of the normalized Laplacian $\mathcal L$ of an Erd\H{o}s-R\'enyi graph $G(n,p)$. To account for the presence of isolated vertices in the sparse regime, we define $\mathcal L$ using the Moore-Penrose pseudoinverse of the degree matrix. Under this convention, we show that the empirical spectral distribution of a suitably normalized $\mathcal L$ converges weakly in probability to the semicircle law whenever $np\to\infty$, thereby providing a rigorous justification of a prediction made in (Akara-pipattana and Evnin, 2023). Moreover, if $np>\log n+\omega(1)$, so that $G(n,p)$ has no isolated vertices with high probability, the same conclusion holds for the standard definition of $\mathcal L$. We further strengthen this result to almost sure convergence when $np=\Omega(\log n)$. Finally, we extend our approach to the Chung-Lu random graph model, where we establish a semicircle law for $\mathcal L$  itself, improving upon (Chung, Lu, and Vu 2003), which obtained the semicircle law only for a proxy matrix.
\end{abstract}

\end{frontmatter}

\section{Introduction}

Spectral questions for random graphs lie at the interface of random matrix theory and random graph theory, and play a central role in a wide range of applications, including high-dimensional statistics \cite{le2018concentration} and theoretical computer science \cite{hoory2006expander}.  Given a graph, the eigenvalues and eigenvectors of its associated matrices control fundamental structural and algorithmic phenomena \cite{chung1997spectral} such as expansion and connectivity, mixing and hitting times of random walks, and the performance of spectral algorithms for clustering and community detection. A canonical and extensively studied model is the Erd\H{o}s--R\'enyi graph $G(n,p)$, whose adjacency matrix provides a prototypical example of a sparse random matrix. Over the past decades, global and local eigenvalue statistics, spectral norm bounds, and eigenvector properties for adjacency matrices have been developed across a broad range of sparsity regimes; see, e.g., \cite{erdHos2013spectral,tran2013sparse,benaych2019largest,alt2021extremal} and references therein.

In spectral graph theory \cite{chung1997spectral}, two particularly important operators are the (combinatorial) Laplacian
$\mathbf{L} \;=\; \mathbf{D}-\mathbf{A}$,
and the normalized Laplacian
\begin{equation}\label{eq:nlm}
\mathbf{\mathcal L} \;=\; \mathbf{I} - \mathbf{D}^{-1/2}\mathbf{A}\mathbf{D}^{-1/2},
\end{equation}
where $\mathbf{A}$ is the adjacency matrix and $\mathbf{D}$ is the diagonal degree matrix. These matrices are ubiquitous in both theory and algorithms: $\mathbf{L}$ is the basic object for cuts, connectivity, and electrical network interpretations \cite{lyons2017probability}, while $\mathbf{\mathcal L}$ (and the closely related random-walk matrix $\mathbf{D}^{-1}\mathbf{A}$) governs random-walk behavior and underpins spectral clustering \cite{meilua2001random}. Compared with adjacency matrices, however, the spectral theory of Laplacian-type matrices for random graphs is less complete, especially in regimes where degree fluctuations are significant. The normalization in \eqref{eq:nlm} introduces a random and dependent reweighting through $\mathbf{D}^{-1/2}$, creating challenges that do not arise in the spectral analysis of $\mathbf{A}$.

\subsection*{Laplacian random matrices}
For the Laplacian $\mathbf{L}$ of an Erd\H{o}s--R\'enyi graph, the empirical spectral distribution (after appropriate centering and scaling) has been shown to converge to the free convolution of a Gaussian distribution with the semicircle law as long as $np\to\infty$; see, e.g., \cite{jiang2012empirical,ding2010spectral,huang2020spectral,chakrabarty2021spectra,chatterjee2022spectral,campbell2024spectrum}.  Beyond global limits, the behavior of extreme eigenvalues (or singular values) has also attracted considerable attention, with recent works establishing sharp asymptotics and fluctuation results in various regimes \cite{campbell2025extreme,ducatez2025spectral,christoffersen2024eigenvalue,subhra2025large}. Related deformed Laplacian matrices, including the Bethe-Hessian, have been studied in the context of community detection in stochastic block models \cite{saade2014spectral,dall2021unified,stephan2024community}. For directed sparse random graphs, limiting spectral laws and related phenomena have also been investigated; see \cite{bordenave2014spectrum}.

\subsection*{The normalized Laplacian} 
A useful perspective on $\mathbf{\mathcal L}$ comes from random walks: writing
\[
\mathbf{\mathcal M} \;:=\; \mathbf{D}^{-1/2}\mathbf{A}\mathbf{D}^{-1/2},
\qquad
\mathbf{\mathcal L} \;=\; \mathbf{I}-\mathbf{\mathcal M},
\]
the matrix $\mathbf{\mathcal M}$ is a symmetric (reversible) Markov kernel obtained by symmetrizing the random-walk matrix $\mathbf{D}^{-1}\mathbf{A}$.   Since the spectra of $\mathcal L$  and $\mathcal M$ have a deterministic relation, we will focus on the spectrum of $\mathcal M$ below.

In dense regimes, since $\mathbf D$ is concentrated around $(np) \mathbf{I}$, the semicircular law for $\mathcal M$ is expected. When $np=\omega(\log n)$, Jiang \cite{jiang2012empirical} showed that the empirical spectral distribution of $\sqrt{np}\mathcal M$ converges almost surely to the semicircle law. In contrast, in the bounded expected degree regime,  for a regularized version of $\mathcal M$, Chi \cite{chi2016random} identified a non-degenerate symmetric limiting distribution characterized by return probabilities on Galton--Watson trees for $G(n,c/n)$. 

In this paper, we bridge the gap between the constant expected-degree regime $(np=c)$ and the supercritical regime $(np=\omega(\log n))$. We show that, as long as $np\to\infty$, the semicircle law for $\mathcal M$ continues to hold, provided one adopts a convention for $\mathbf{D}^{-1/2}$ that remains well defined in the presence of isolated vertices (equivalently, by projecting away zero-degree coordinates). In view of the result of Chi \cite{chi2016random}, the condition $np\to\infty$ is both necessary and sufficient for the semicircle law to hold. The semicircle law for $\mathcal M$ when $np\to\infty$ was also conjectured by Akara-pipattana and Evnin in \cite[Section 5.2]{akara2023random} based on predictions from the Fyodorov--Mirlin method, and we provide a rigorous justification.

Complementing asymptotic results, a substantial line of work establishes non-asymptotic extreme eigenvalue bounds of  $\mathbf{\mathcal L}$, motivated in part by stability guarantees for spectral clustering \cite{le2017concentration,deng2021strong,chen2025concentration}; see also \cite{hoffman2021spectral} for a sharp spectral-gap estimate restricted to the giant component of $G(n,p)$.
Beyond random graph models, the limiting spectrum of the random walk matrix $\mathbf{D}^{-1}\mathbf{A}$ has also been considered for more general Hermitian \cite{bordenav2011spectrum,bordenave2010spectrum} or non-Hermitian random matrix $\mathbf A$ \cite{bordenave2012circular}.



\subsection*{The Chung-Lu model}

Going beyond homogeneous Erd\H{o}s--R\'enyi graphs, inhomogeneous random graphs $G(n,p_{n,ij})$, where edges $(i,j)$ are present independently with probabilities $p_{n,ij}$, provide a flexible modeling framework that captures heterogeneous expected degrees and latent structure. A prominent special case is the rank-one (Chung-Lu) model \cite{chung2002average,chung2006complex}, which prescribes expected degrees via a factorized probability matrix. See Section~\ref{sec:CL} for a detailed description of the model.  In this setting, Chung, Lu, and Vu \cite{chung2003spectra} established a semicircle law for a matrix $\overline{\mathcal M}=\mathbb E[\mathbf{D}]^{-1/2}\mathbf{A} \mathbb E[\mathbf{D}]^{-1/2}$ approximating the normalized Laplacian, and further spectral estimates of $\mathcal L$ were developed in \cite{lu2013spectra}. Local spectral statistics for its adjacency matrix were investigated in \cite{adlam2015spectral}. 

We extend the semicircle law of the normalized Laplacian for $G(n,p)$ to the Chung--Lu model. In particular, we show that after normalization,  a semicircle law holds for $\mathcal M$ itself,  rather than the proxy matrix $\overline{\mathcal M}$ studied in \cite{chung2003spectra}. 

Beyond the rank-one setting, for general inhomogeneous random graphs, the limiting spectral distribution of $\mathcal M$ need not be semicircular, even in dense regimes. If one further assumes that the variance profile of $\overline{\mathcal M}$ admits a graphon limit \cite{zhu2020graphon}, the comparison argument developed in this paper can still be applied to characterize the corresponding limiting spectral distribution of $\mathcal L$.


\subsection*{Notation}
Let $(\mathbb{R}, \mathcal{B})$ denote the measurable space, where $\mathcal{B}$ is the Borel $\sigma$-algebra on $\mathbb{R}$.  
The indicator function is denoted by $\mathbf{1}(\cdot)$, and $\delta_x$ stands for the Dirac measure concentrated at $x \in \mathbb{R}^d$.  
The constant $c$ may change from line to line.  We also use the notation $[n] := \{1,2,\dots,n\}$.
For two sequences of nonnegative real numbers $\{a_n\}$ and $\{b_n\}$, we write $a_n = o(b_n)$ if $\lim_{n \to \infty} a_n / b_n = 0$, and $a_n = O(b_n)$ if $a_n \le C b_n$ for some $C > 0$ and all large $n$. We write $a_n\gg b_n$ if $a_n/b_n\to\infty$.

  We denote by $\mu_{\mathrm{sc}}$ the semicircle distribution supported on $[-2,2]$, with density function
    \begin{equation}\label{sem}
       d\mu_{\mathrm{sc}}=\frac{1}{2\pi}\sqrt{4-x^2}\mathbf{1}\{|x|\le 2\}dx.
    \end{equation}

\subsection*{Organization of the paper} The remainder of this paper is organized as follows. We state our main results for the normalized Laplacian of $G(n,p)$ and the Chung-Lu model in Section~\ref{sec:main}. Section~\ref{sec:prelim} includes some basic facts about random matrix theory. We prove Theorems~\ref{thm2} and \ref{thm1} in Section~\ref{sec:proof_main} and the proof of Theorem~\ref{thm:CLV} is given in Section \ref{sec:proof_CLV}.

\section{Main Results}\label{sec:main}

\subsection{The Erd\H{o}s-R\'{e}nyi graph $G(n,p_n)$}

 Let $\mathbf{A}_n=(a_{ij}^{(n)})$ be the adjacency matrix of an Erd\H{o}s-R\'{e}nyi graph $G(n,p_n)$ such that $a_{ii}^{(n)}=0, i\in [n]$.
The normalized Laplacian matrix is defined by 
\begin{equation}\label{nlm}
\mathbf{\mathcal L}_n=\mathbf{I}_n-\mathbf{D}_n^{-1/2}\mathbf{A}_n\mathbf{D}_n^{-1/2},
\end{equation}
 where $ \mathbf{D}_{n}=\operatorname{diag}(\sum_{j \neq i}^{n} a_{i j}^{(n)})_{1 \leq i \leq n} $ is a diagonal matrix. 
 If a graph has isolated vertices, $\mathbf{D}_n$ is not invertible. Still, we can extend the definition of $\mathbf{D}_n^{-1}$ using the Moore-Penrose pseudoinverse \cite{van1996matrix} by defining $\mathbf{D}_n^{-1}(i,i)=0$ if $\mathbf{D}_n(i,i)=0$.  

\begin{theorem}[Semicircle law for $\mathcal L_n$]\label{thm2}
Assume that $D_n^{-1}$ is defined using the Moore--Penrose
pseudoinverse convention. Suppose that there exists
$\varepsilon_0>0$ such that $p_n \le 1-\varepsilon_0$ 
for all sufficiently large $n$. Then the empirical spectral distribution of $\sqrt{\frac{np_n}{1-p_n}}\,(I_n-\mathcal L_n)$ converges weakly in probability to the semicircle law
$\mu_{\mathrm{sc}}$.
\end{theorem}

\begin{remark}[Optimality]
When $np_n=c$, \cite[Theorem~5]{chi2016random} showed that the limiting spectral distribution of $\mathbf I_n-\mathcal L_n$ is not semicircular. Consequently, the condition $np_n\to\infty$ is optimal for the semicircle law to hold, improving upon the stronger assumption $np_n=\omega(\log n)$ required in \cite{jiang2012empirical}. We note that Theorem~\ref{thm2} was also conjectured in \cite[Section 5]{akara2023random} based on predictions from the Fyodorov--Mirlin method.
\end{remark}

\begin{remark}
  The assumption that $p_n$ is bounded away from $1$ excludes the
near-complete regime. For example, let $p_n = 1-n^{-3}.$  Then by a union bound,
$G(n,p_n)$ is the complete graph
with probability tending to one, and its limiting spectral distribution is not the semicircle law.

\end{remark}

Since $\mathbf{D}_n$ is invertible if and only if $G(n,p_n)$ has no isolated vertices, which occurs with high probability whenever $np_n>\log n+\omega(1)$ (see, e.g., \cite[Theorem~7.3]{bollobas2011random}), we obtain the following corollary of Theorem~\ref{thm2}:
\begin{corollary}
Assume that $\sup_n p_n<1-\varepsilon_0$ and  there exists a sequence $c_n\to\infty$ such that
\[n p_n \ge  \log n+c_n.\]
Then, with probability tending to one, the degree matrix $\mathbf{D}_n$ is invertible.
Moreover, the empirical spectral distribution of $\sqrt{\frac{np_n}{1-p_n}} (\mathbf{I}_n-\mathcal L_n)$
converges weakly in probability to the semicircle law as $n \to \infty$.
\end{corollary}

By further assuming that $np_n=\Omega(\log n)$, we can strengthen convergence in probability to almost sure convergence, as stated below:
\begin{theorem}\label{thm1}
Suppose that  $\sup_n p_n \le 1 - \varepsilon_0$, and 
$n p_n \ge C \log n$ for some absolute constant $C$.
Then, almost surely, the empirical spectral distribution of $\sqrt{\frac{np_n}{1-p_n}} (\mathbf{I}_n-\mathcal L_n)$ converges weakly to the semicircle law.
\end{theorem}
A careful examination of the proof indicates that Theorem~\ref{thm1} holds for $C>36$, although no attempt is made to optimize the numerical constant.

\subsection{The Chung--Lu Model}\label{sec:CL}

The Chung--Lu model \cite{chung2002average,chung2006complex} is a widely used random graph model for generating random graphs with a prescribed expected degree sequence.
It can be viewed as a natural generalization of the Erd\H{o}s--R\'enyi random graph.
Let $n \in \mathbb{N}$ and let
\[
\mathbf{w}_n = (w_{n,1}, w_{n,2}, \dots, w_{n,n})
\]
be a sequence of positive weights.
Define
\[
\phi_n := \frac{1}{\sum_{i=1}^n w_{n,i}} .
\]
The Chung--Lu random graph $G(n,\mathbf{w}_n)$ is an undirected random graph on the vertex set
$[n] := \{1,2,\dots,n\}$,
in which edges between distinct vertices $i \neq j$ are placed independently with probabilities
\[
\mathbb{P}\!\left( a_{ij}^{(n)} = 1 \right)
= p_{n,ij}
:= w_{n,i} w_{n,j} \phi_n ,
\qquad
\mathbb{P}\!\left( a_{ij}^{(n)} = 0 \right)
= 1 - p_{n,ij}.
\]
Here $a_{ij}^{(n)}$ denotes the adjacency indicator between vertices $i$ and $j$.


Let $\mathbf{A}_n = (a_{ij}^{(n)})_{1 \le i,j \le n}$ denote the adjacency matrix of $G(n,\mathbf{w}_n)$, and let
$\mathbf{D}_n = \operatorname{diag}(d_{n,1},\dots,d_{n,n})$ be the corresponding degree matrix.
We further denote by $\mathbf{W}_n$ the diagonal matrix with diagonal entries $w_{n,i}$, representing the expected degrees of the vertices.
A fundamental property of the Chung-Lu model is that the expected degree of vertex $i$ satisfies
\[
\mathbb{E}[d_{n,i}]
= \sum_{j=1}^n p_{n,ij}
\approx w_{n,i},
\]
provided that
\[
\max_{1 \le i \le n} w_{n,i}^2 \;<\; \sum_{j=1}^n w_{n,j},
\]
which ensures that $p_{n,ij} \le 1$ for all $i,j$.
Consequently, the weight $w_{n,i}$ can be interpreted as the target expected degree of vertex $i$.
    As a special case, if $w_{n,i} \equiv n p_n$ for all $i$, then $p_{n,ij} = p_n$, and the Chung-Lu model reduces to the classical Erd\H{o}s--R\'enyi random graph $G(n,p_n)$.

We now state our main result under the Chung-Lu model.

\begin{theorem}[Semicircle law for  the Chung-Lu model]\label{thm:CLV}
Consider the Chung-Lu random graph $G(n,\mathbf{w}_n)$ defined above, and let
\[
w_{n,\min} := \min_{1 \le i \le n} w_{n,i},
\qquad
\bar{w}_n := \frac{1}{n} \sum_{i=1}^n w_{n,i}.
\]
Assume that
\begin{equation}\label{clvcondition}
    w_{n,\min} \gg \sqrt{\bar{w}_n}
\end{equation}
and $\bar w_n=o(n)$. Then the empirical spectral distribution of $\sqrt{\bar{w}_n} (\mathbf{I}_n-\mathcal L_n)$ converges weakly in probability to the semicircle law.

\end{theorem}
\begin{remark}[Comparison with \cite{chung2003spectra}]
  Condition~\eqref{clvcondition} implies that $w_{n,\min}\to\infty$ as $n\to\infty$. Under the same assumptions, Chung, Lu, and Vu \cite[Theorem 6]{chung2003spectra} established a semicircle law only for the proxy matrix $\mathbf{W}_n^{-1/2}\mathbf{A}_n\mathbf{W}_n^{-1/2}$. The present work closes this gap by establishing the semicircle law for the normalized Laplacian itself.

 In the homogeneous sparse case $w_{n,i}\equiv d_n$ with
$d_n=o(n)$, condition~\eqref{clvcondition} reduces to $d_n\to\infty$, which agrees
with the optimal threshold in the Erdős--Rényi setting. For a
general heterogeneous weight sequence, however, the use of the
minimum expected degree is unlikely to be sharp. In particular,
condition~\eqref{clvcondition} excludes even a vanishing number of low-weight
vertices, although deleting $o(n)$ such vertices does not change the limiting
spectral distribution. Thus, a
natural direction for further work is to replace the minimum degree
condition by an assumption on the lower tail of the empirical
weight distribution.
\end{remark}

\begin{remark}
\label{rem:clv-sparsity}
We note that an additional sparsity assumption $\bar w_n=o(n)$ appears to have been
left implicit in ~\cite{chung2003spectra}.
Indeed, for the centered proxy matrix with entries
\[
    C_{n,ij}
    =
    \frac{a_{ij}^{(n)}-p_{n,ij}}
         {\sqrt{w_{n,i}w_{n,j}}},
\]
one has
\[
    \mathbb{E}C_{n,ij}^{\,2}
    =
    \phi_n(1-p_{n,ij}),
\]
rather than \(\phi_n\). Thus, the leading  terms in the moment
expansion contain additional factors \(1-p_{n,ij}\), which cannot in
general be replaced by \(1\) under the condition
\(w_{n,\min}\gg\sqrt{\bar w_n}\) alone. For example, in the homogeneous
case \(w_{n,i}\equiv np_n\) with \(p_n\to p\in(0,1)\), the normalization
\(\sqrt{\bar w_n}\) yields a semicircle law with variance \(1-p\), rather
than the standard semicircle law.
\end{remark}

\section{Preliminaries}\label{sec:prelim}

Let $\mathbf{M}$ be a symmetric matrix with eigenvalues 
$\lambda_1(\mathbf{M}) \ge \lambda_2(\mathbf{M}) \ge \cdots \ge \lambda_n(\mathbf{M})$.
The empirical spectral distribution (ESD) of $\mathbf{M}$, denoted by $\hat{\mu}(\mathbf{M})$, is defined as
\[
\hat{\mu}(\mathbf{M}) = \frac{1}{n} \sum_{j=1}^{n} \delta_{\lambda_j(\mathbf{M})}.
\]
If $\mathbf{M}$ is random, then $\hat{\mu}(\mathbf{M})$ is a random probability measure on $(\mathbb{R}, \mathcal{B})$.
We write $\hat{\mu}(\mathbf{M}_n) \Rightarrow \mu$ to denote that $\hat{\mu}(\mathbf{M}_n)$ converges weakly to a probability measure $\mu$ on $(\mathbb{R}, \mathcal{B})$.

\begin{definition}[bounded Lipschitz metric]

Let $d_{\text{BL}}$ denote the bounded Lipschitz metric
\begin{equation}
d_{\text{BL}}(\mu, \nu) = \sup \left\{ \int f \, d\mu - \int f \, d\nu : \| f \|_\infty + \| f \|_L \le 1 \right\},
\end{equation}
where $\| f \|_\infty = \sup\limits_x |f(x)|$, $\| f \|_L = \sup\limits_{x \ne y} \dfrac{|f(x) - f(y)|}{|x - y|}  $.
\end{definition}

We will use the following comparison inequality to establish weak convergence of empirical spectral distributions: 
\begin{lemma}
For any symmetric matrices $A, B$, we have 
 \begin{equation}\label{dbl}
d_{\text{BL}}^2(\hat{\mu}(\mathbf{A}), \hat{\mu}(\mathbf{B})) \le \frac{1}{n} \operatorname{tr}((\mathbf{B} - \mathbf{A})^2).
\end{equation}
\end{lemma}

\begin{proof}
    For the spectral measures of $n \times n$ symmetric real matrices $\mathbf{A}, \mathbf{B}$ we have
\begin{align*}
    d_{\text{BL}}(\hat{\mu}(\mathbf{A}), \hat{\mu}(\mathbf{B})) &\le 
\sup \left\{ \frac{1}{n} \sum_{j=1}^n | f(\lambda_j(\mathbf{A})) - f(\lambda_j(\mathbf{B})) | : \| f \|_L \le 1 \right\}\\
&\le \frac{1}{n} \sum_{j=1}^n | \lambda_j(\mathbf{A}) - \lambda_j(\mathbf{B}) |\\
&\le \frac{1}{\sqrt n} \sqrt{\sum_{j=1}^n | \lambda_j(\mathbf{A}) - \lambda_j(\mathbf{B}) |^2} \le \frac{1}{\sqrt n} \sqrt{\operatorname{tr}((\mathbf{B} - \mathbf{A})^2)},
\end{align*}
where the third inequality follows from the Cauchy--Schwarz inequality,
and the last inequality is the Hoffman--Wielandt inequality (see, e.g., \cite[Lemma 2.3]{bai1999methodologies}). Squaring both sides completes the proof.
\end{proof}

\section{Proof of Theorems~\ref{thm2} and \ref{thm1}} 
\label{sec:proof_main}

Before proceeding, we briefly outline the overall proof strategy.
The proof is divided into several steps. From Ding and Jiang \cite{ding2010spectral}, the empirical spectral distribution of 
\(
\mathbf{A}_n/\sqrt{n p_n (1-p_n)}
\)
converges weakly to the semicircle law when $np_n(1-p_n)\to\infty$.
Therefore, it suffices to show that the bounded Lipschitz distance between the
empirical spectral distribution under consideration, and that of
$\mathbf{A}_n/\sqrt{n p_n (1-p_n)}$
converges to zero almost surely as $n \to \infty$.
A direct control of $d_{\text{BL}}$ is technically challenging.
To address this issue, we decompose the analysis into two parts:
the expectation term and the deviation from its expectation.
The expectation is handled in Lemma~\ref{lem1}, where we derive a suitable tail estimate,
while the fluctuation around the expectation is controlled in Lemma~\ref{lem2}
by an improved bounded differences inequality.

Lemma~\ref{lem1} is first used to establish Theorem~\ref{thm2}.
Then, by combining Lemmas~\ref{lem1} and~\ref{lem2}, we show that the bounded Lipschitz
distance converges to zero almost surely, which in turn yields the desired spectral
convergence and completes the proof of Theorem~\ref{thm1}.

\medskip 

Set 
\begin{equation}\label{def}
    \widetilde{\mathbf{\mathcal L}}_n=\mathbf{I}_n-\frac{1}{(n-1)p_n}\mathbf{A}_n ,\ u_n=(n-1)p_n , \ d_{n,i}=\sum\limits_{j \ne i}a_{ij}^{(n)}.
    \end{equation}
    
    Note that\begin{align*}
\widetilde{\mathbf{\mathcal L}}_n - \mathbf{\mathcal L}_n
&= -\frac{1}{(n-1)p_n} \mathbf{A}_n
+ \operatorname{diag}(d_{n,1}^{-1/2}, \ldots, d_{n,n}^{-1/2}) 
  \mathbf{A}_n 
  \operatorname{diag}(d_{n,1}^{-1/2}, \ldots, d_{n,n}^{-1/2}) \\[6pt]
&= \left( 
  \left( 
    -\frac{1}{(n-1)p_n} 
    + \frac{1}{\sqrt{d_{n,i}} \sqrt{d_{n,j}}}
  \right)
  a^{(n)}_{ij} 
  \right)_{n \times n}.
\end{align*}

\begin{lemma}[Approximation of $\frac{1}{n}\mathbb{E}(\operatorname{tr}(\widetilde{\mathbf{\mathcal L}}_n - \mathbf{\mathcal L}_n)^2)$]\label{lem1}
    Assume that the matrices $\mathbf{\mathcal L}_n$ and $\widetilde{\mathbf{\mathcal L}}_n$ are defined as \eqref{nlm} and \eqref{def}. If  $np_n\to \infty$, then  $$\frac{1}{n}\mathbb{E}(\operatorname{tr}(\widetilde{\mathbf{\mathcal L}}_n - \mathbf{\mathcal L}_n)^2)=O( u_n^{-2}).$$
\end{lemma}

\begin{lemma}[Concentration]\label{lem2}
    Assume that the matrices $\mathbf{\mathcal L}_n$ and $\widetilde{\mathbf{\mathcal L}}_n$ are defined as \eqref{nlm} and \eqref{def} and that
\begin{equation}\label{lem3-cond}
    n p_n \ge C_0 \log n
\end{equation}
for all sufficiently large $n$, where $C_0>0$ is a sufficiently large absolute constant. Then, almost surely, $$\frac{1}{n}\operatorname{tr}(\widetilde{\mathbf{\mathcal L}}_n - \mathbf{\mathcal L}_n)^2-\frac{1}{n}\mathbb{E}(\operatorname{tr}(\widetilde{\mathbf{\mathcal L}}_n - \mathbf{\mathcal L}_n)^2)= o( u_n^{-1}).$$ 
\end{lemma}

\subsection{Proof of Lemma~\ref{lem1}}

    Through symmetry and the fact that $\{a_{ij}^{(n)}\}$ are i.i.d. random variables, we can derive that
    \begin{equation}\label{etr}
    \begin{aligned}
    \frac{1}{n}\mathbb{E}\!\left[\operatorname{tr}\!\left(\widetilde{\mathbf{\mathcal L}}_n - \mathbf{\mathcal L}_n\right)^2\right]
    &= \frac{1}{n}\mathbb{E}\!\left[\sum_{i \ne j}^n 
    \left(-\frac{1}{u_n} + \frac{1}{\sqrt{d_{n,i}}\sqrt{d_{n,j}}}\right)^2 
    a_{ij}^{(n)}\right] \\
    &= \frac{1}{n} \sum_{i \ne j}^n 
    \mathbb{E}\!\left[\left(-\frac{1}{u_n} + \frac{1}{\sqrt{d_{n,i}}\sqrt{d_{n,j}}}\right)^2 
    a_{ij}^{(n)}\right] \\
    &= \frac{1}{n} \cdot n(n - 1) 
    \mathbb{E}\!\left[\left(-\frac{1}{u_n} + \frac{1}{\sqrt{d_{n,1}}\sqrt{d_{n,2}}}\right)^2 
    a_{12}^{(n)}\right]\\
     &= (n - 1) 
    \mathbb{E}\!\left[\left(-\frac{1}{u_n} + \frac{1}{\sqrt{d_{n,1}}\sqrt{d_{n,2}}}\right)^2 
    a_{12}^{(n)}\right].
\end{aligned}\end{equation}

    Note that $\mathbb{P}(a_{12}^{(n)}=1)=1-\mathbb{P}(a_{12}^{(n)}=0)=p_n$,  we get $$\mathbb{E}\!\left[\left(-\frac{1}{u_n} + \frac{1}{\sqrt{d_{n,1}}\sqrt{d_{n,2}}}\right)^2 
    a_{12}^{(n)}\right]=p_n\mathbb{E}\!\left[\left(-\frac{1}{u_n} + \frac{1}{\sqrt{d_{n,1}}\sqrt{d_{n,2}}}\right)^2 
    |a_{12}^{(n)}=1\right].$$
    
    We briefly outline our approach for evaluating 
\[
\mathbb{E}\!\left[\left(-\frac{1}{u_n} + \frac{1}{\sqrt{d_{n,1}d_{n,2}}}\right)^2 
\,\big|\, a_{12}^{(n)} = 1 \right].
\]

We introduce an event $E_1$ under which both $d_{n,1}$ and $d_{n,2}$ remain close to $u_n$, so that the function $(x,y)\mapsto \left(\frac{1}{u_n}-\frac{1}{\sqrt{xy}}\right)^2$ can be locally approximated. 
The expectation is then decomposed into two parts corresponding to $E_1$ and its complement $E_1^c$. 
For $E_1$, we use independence and moment estimates of binomial variables to show that the contribution is $O(u_n^{-3})$, while for $E_1^c$, a Chernoff bound ensures that its probability is exponentially small. 
Combining these estimates yields the desired order of magnitude for the entire expectation.

        Given that $a_{12}^{(n)} = 1$, let
    \[
    d_{n,1} =1+Z_1 := 1 + \sum_{k \ne 1,2} a_{1k}^{(n)}, 
\qquad 
d_{n,2} =1+Z_2:= 1 + \sum_{k \ne 1,2} a_{2k}^{(n)}.
\]
So we have
\[
Z_1 \sim \mathrm{Bin}(n-2, p_n)
\quad \text{and} \quad
Z_2 \sim \mathrm{Bin}(n-2, p_n),
\]
where $\mathrm{Bin}(n, p)$ denotes the binomial distribution with parameters $n$ and $p$. Thus,
\[
\mathbb E(Z_1+1)=\mathbb E(Z_2+1)
=(n-2)p_n+1
=u_n+(1-p_n)
=u_n+O(1),
\]
and
\[
\mathbb{E}\left(\left(Z_1+1\right)^2\right)=\mathbb{E}\left(\left(Z_2+1\right)^2\right)=3(n-2)p_n+(n-2)(n-3)p_n^2+1=u_n^2+O(u_n).
\]

Let the event $$E_1:=\{d_{n,1},d_{n,2}\in[u_n/2,3u_n/2]\}.$$ 

We  first approximate $$\mathbb{E}\!\left(\left(-\frac{1}{u_n} + \frac{1}{\sqrt{d_{n,1}}\sqrt{d_{n,2}}}\right)^2\mathbf{1}(E_1) 
    |a_{12}^{(n)}=1\right).$$

    Using the identity  $|\frac{1}{\sqrt{x}}-\frac{1}{\sqrt{y}}|=|\frac{\sqrt{y}-\sqrt{x}}{\sqrt{xy}}|=|\frac{y-x}{\sqrt{xy}(\sqrt{x}+\sqrt{y})}|\le|\frac{y-x}{y\sqrt{x}}|$ for $x=(Z_1+1)(Z_2+1) \ ,y=u_n^2 $, we obtain
    $$\begin{aligned}
       \mathbb{E}\!\left(\left(-\frac{1}{u_n} + \frac{1}{\sqrt{d_{n,1}}\sqrt{d_{n,2}}}\right)^2\mathbf{1}(E_1) 
    |a_{12}^{(n)}=1\right)
    =&\mathbb{E}\!\left(\left(-\frac{1}{u_n} + \frac{1}{\sqrt{Z_1+1}\sqrt{Z_2+1}}\right)^2\mathbf{1}(E_1) 
    \right)\\
   \le &\mathbb{E}\!\left(\left(\frac{(Z_1+1)(Z_2+1)-u_n^2}{u_n^2\sqrt{Z_1+1}\sqrt{Z_2+1}}\right)^2\mathbf{1}(E_1)\right).\\
    \end{aligned}$$

    By non-negativity, note that $((Z_1+1)(Z_2+1)-u_n^2)^2\mathbf{1}(E_1)\le ((Z_1+1)(Z_2+1)-u_n^2)^2$,and when event $E_1$ holds, $(Z_i+1)^{-1}\le 2/u_n$ for $i=1,2$. Since $Z_1$ and $Z_2$ are independent, we can therefore bound the expectation by  
    {\small\begin{equation}\label{ec}
        \begin{aligned}
       \mathbb{E}\!\left(\left(\frac{(Z_1+1)(Z_2+1)-u_n^2}{u_n^2\sqrt{Z_1+1}\sqrt{Z_2+1}}\right)^2\mathbf{1}(E_1)\right)
       \le&\mathbb{E}\!\left(\frac{4((Z_1+1)(Z_2+1)-u_n^2)^2}{u_n^6}\right)\\
        =&\frac{4}{u_n^6}\cdot \mathbb{E}\!\left(\left(Z_1+1\right)^2\left(Z_2+1\right)^2-2u_n^2(Z_1+1)(Z_2+1)+u_n^4\right)\\
        =&\frac{4}{u_n^6}\cdot \left(\mathbb{E}\!\left(\left(Z_1+1\right)^2\right)\mathbb{E}\!\left(\left(Z_2+1\right)^2\right)-2u_n^2\mathbb{E}\!\left(Z_1+1\right)\mathbb{E}\!\left(Z_2+1\right)+u_n^4\right)\\
        =&\frac{4}{u_n^6}(u_n^4+O(u_n^3)-2u_n^4+O(u_n^3)+u_n^4)\\
        =&O(u_n^{-3}).
    \end{aligned}
    \end{equation}}


Under the assumption $n p_n  \to \infty$, a direct application of Chernoff’s bound yields
\begin{equation}\label{cb}
\mathbb{P}(E_1^c \mid a_{12}=1)
\le  2 \exp\!\left(-u_n/12\right).
\end{equation}

When event $E_1^c$ occurs, observe that $$|\frac{1}{u_n}-\frac{1}{\sqrt{d_{n,1}}\sqrt{d_{n,2}}}| \le |\frac{1}{u_n}|+|\frac{1}{\sqrt{d_{n,1}}\sqrt{d_{n,2}}}|\le 1+1=2,$$ 
where we use the pseudo-inverse convention of $\mathbf{D}_n^{-1}$.

Thus we have \begin{equation}\label{e1c}
	\mathbb{E}\!\left(\left(-\frac{1}{u_n} + \frac{1}{\sqrt{d_{n,1}}\sqrt{d_{n,2}}}\right)^2\mathbf{1}(E_1^C) 
	|a_{12}^{(n)}=1\right)\le 4\mathbb{P}(E_1^c)\le8\exp\!\left(-u_n/12\right).
\end{equation}

Since $u_n \to \infty$ as $n \to \infty$, it follows that
\[
\exp\!\left(-u_n/12\right) = o\!\left(u_n^{-3}\right).
\]

Combine \eqref{ec} and \eqref{e1c}, we get
\begin{equation}\label{eqcoombine}
\begin{aligned}
	\mathbb{E}\!\left(\left(-\frac{1}{u_n} + \frac{1}{\sqrt{d_{n,1}}\sqrt{d_{n,2}}}\right)^2 
	|a_{12}^{(n)}=1\right)
    =&\mathbb{E}\!\left(\left(-\frac{1}{u_n} + \frac{1}{\sqrt{d_{n,1}}\sqrt{d_{n,2}}}\right)^2\left(\mathbf{1}(E_1)+\mathbf{1}(E_1^c)\right) 
	|a_{12}^{(n)}=1\right)\\
    =& O(u_n^{-3}).\\
\end{aligned}
\end{equation}

Substituting \eqref{eqcoombine} into \eqref{etr}, we obtain
     \begin{equation}
         \frac{1}{n}\mathbb{E}\!\left[\operatorname{tr}\!\left(\widetilde{\mathbf{\mathcal L}}_n - \mathbf{\mathcal L}_n\right)^2\right]=(n-1)p_n \cdot O(u_n^{-3})=O(u_n^{-2}).
     \end{equation}

     \subsection{Proof of Theorem~\ref{thm2}}
We begin by comparing the empirical spectral distributions (ESDs) of 
$\mathbf{\mathcal L}_n$ and $\widetilde{\mathbf{\mathcal L}}_n$. Since the asymptotic spectral behavior of $\widetilde{\mathbf{\mathcal L}}_n$ has already been established in Ding and Jiang \cite{ding2010spectral}, 
it is enough to show that
\[
d_{\mathrm{BL}}\!\left(\hat{\mu}(\mathbf{\mathcal L}_n), \hat{\mu}(\widetilde{\mathbf{\mathcal L}}_n)\right)
\longrightarrow 0
\quad \text{in probability as } n \to \infty.
\]

By Lemma~\ref{lem1}, we have
\begin{equation}\label{MESD}
    \mathbb{E}\!\left[
d_{\mathrm{BL}}^2\!\left(\hat{\mu}(\widetilde{\mathbf{\mathcal L}}_n),\hat{\mu}(\mathbf{\mathcal L}_n)\right)
\right]
\le \frac{1}{n}\,
\mathbb{E}\!\left[
\operatorname{tr}\!\left(\widetilde{\mathbf{\mathcal L}}_n - \mathbf{\mathcal L}_n\right)^2
\right]
= O(u_n^{-2}).
\end{equation}

By Markov’s inequality, \eqref{MESD} implies that the ESDs of $\mathbf{\mathcal L}_n$ and $\widetilde{\mathbf{\mathcal L}}_n$ are asymptotically indistinguishable in the $d_{\mathrm{BL}}$ metric. The same argument applies to scaled and shifted versions of $\mathbf{\mathcal L}_n$ and $\widetilde{\mathbf{\mathcal L}}_n$. 
Indeed, for any sequences of constants $\{\gamma_n\}$ and $\{\rho_n\}$,
\begin{equation}
\mathbb{E}\!\left[
d_{\mathrm{BL}}^2\!\left(
\hat{\mu}(\gamma_n \mathbf{\mathcal L}_n + \rho_n \mathbf{I}_n), 
\hat{\mu}(\gamma_n \widetilde{\mathbf{\mathcal L}}_n + \rho_n \mathbf{I}_n)
\right)
\right]
= O\!\left(\frac{\gamma_n^2}{u_n^2}\right).
\end{equation}

Applying Markov’s inequality yields
\[
\mathbb{P}\!\left(
d_{\mathrm{BL}}^2\!\left(
\hat{\mu}(\gamma_n \mathbf{\mathcal L}_n + \rho_n \mathbf{I}_n), 
\hat{\mu}(\gamma_n \widetilde{\mathbf{\mathcal L}}_n + \rho_n \mathbf{I}_n)
\right)
\ge \varepsilon
\right)
\le c\,\frac{\gamma_n^2}{u_n^2\,\varepsilon}.
\]

Now choose \begin{equation}\label{chgr}\gamma_n = -\rho_n = -\sqrt{n p_n /(1-p_n)}. \end{equation}
Then
\[
\frac{\gamma_n^2}{u_n^2}
= O\!\left(\frac{1}{n p_n (1-p_n)}\right)
= o(1),
\]
and hence
\begin{equation}\label{wc1}
\hat{\mu}(\gamma_n \mathbf{\mathcal L}_n + \rho_n \mathbf{I}_n)
\Longrightarrow
\hat{\mu}(\gamma_n \widetilde{\mathbf{\mathcal L}}_n + \rho_n \mathbf{I}_n)
\quad \text{in probability as } n \to \infty.
\end{equation}

 By the result of Ding and Jiang~\cite{ding2010spectral},  we have
\[
\frac{1}{n}\sum_{i=1}^{n}
\mathbf{1}\!\left(
\sqrt{\frac{n p_n}{1-p_n}}
\left(1 - \lambda_i(\widetilde{\mathbf{\mathcal L}}_n)\right)
\le x
\right)
\Longrightarrow \mu_{\mathrm{sc}},
\quad \text{a.s. as } n \to \infty.
\]

With \eqref{chgr}, for all $1 \le i \le n$, 
\[
\lambda_{n-i}\!\left(\gamma_n \widetilde{\mathbf{\mathcal L}}_n + \rho_n \mathbf{I}_n\right)
= \sqrt{\frac{n p_n}{1-p_n}}
\left(1 - \lambda_i(\widetilde{\mathbf{\mathcal L}}_n)\right).
\]
Therefore,
\begin{equation}\label{wc2}
\hat{\mu}(\gamma_n \widetilde{\mathbf{\mathcal L}}_n + \rho_n \mathbf{I}_n)
\Longrightarrow \mu_{\mathrm{sc}}
\quad \text{a.s. as } n \to \infty.
\end{equation}

Finally, combining \eqref{wc1} and \eqref{wc2} yields
\[
\hat{\mu}(\gamma_n \mathbf{\mathcal L}_n + \rho_n \mathbf{I}_n)
\Longrightarrow \mu_{\mathrm{sc}}
\quad \text{in probability as } n \to \infty,
\]
which completes the proof. 
    
     \subsection{Proof of Lemma~\ref{lem2}}
 Firstly, we introduce the following result of Warnke~\cite{warnke2016method}.

\begin{lemma}[Typical bounded differences inequality]
Let $X = (X_1, \ldots, X_N)$ be a family of independent random variables with 
$X_k$ taking values in a set $\Lambda_k$. 
Let $\Gamma \subseteq \prod_{j \in [N]} \Lambda_j$ be an event, 
and assume that the function $F : \prod_{j \in [N]} \Lambda_j \to \mathbb{R}$ 
satisfies the following \emph{typical Lipschitz condition}:

\medskip
\noindent
\textbf{(TL)} There exist sequences $(c_k)_{k \in [N]}$ and $(d_k)_{k \in [N]}$ 
with $c_k \le c_k'$ such that whenever 
$x, \sigma_k(x) \in \prod_{j \in [N]} \Lambda_j$ differ only in the $k$-th coordinate, we have
\[
|F(x) - F(\sigma_k(x))| \le 
\begin{cases}
c_k, & \text{if } x \in \Gamma,\\[3pt]
c_k', & \text{otherwise.}
\end{cases}
\]

\noindent
For any $(\beta_k)_{k \in [N]}$ with $\beta_k \in (0,1]$, there exists an event  
$\mathcal{A} = \mathcal{A}(\Gamma, (\beta_k)_{k \in [N]})$ satisfying
\begin{equation}\label{tbd}
  \mathbb{P}(\mathcal{A}) \le \sum_{k \in [N]} \beta_k^{-1} \, \mathbb{P}(X \notin \Gamma),
  \qquad 
  \neg \mathcal{A} \subseteq \Gamma,
\end{equation}
such that, for $\mu = \mathbb{E}F(X)$, $\varepsilon_k = \beta_k (c_k' - c_k)$, and any $t \ge 0$,
\[
\mathbb{P}(F(X) \ge \mu + t \text{ and } \neg \mathcal{A})
\le \exp\!\left( - \frac{t^2}{2 \sum_{k \in [N]} (c_k + \varepsilon_k)^2} \right).
\]
\end{lemma}

\medskip

\begin{proof}[Proof of Lemma~\ref{lem2}]
    
Consider the function \( f_n : \{0,1\}^{n \times n} \to \mathbb{R} \) defined by  
\begin{equation}\label{fn}
 f_n(\{X_{i,j}\}_{1 \le i,j \le n})
 := \sum_{i \ne j}^n 
 \left( u_n^{-1} - \left( \sum_{k \ne i} X_{ik} \sum_{k \ne j} X_{jk} \right)^{-1/2} \right)^2 X_{ij},
\end{equation}
which depends only on the adjacency variables $\{X_{ij}\}$ of the random graph. Then 
\[F=\frac{1}{n}f_n=\frac{1}{n}\operatorname{tr}(\widetilde{\mathbf{\mathcal L}}_n - \mathbf{\mathcal L}_n)^2.\]








If we change only one coordinate $(u,v)$ of $\{X_{ij}\}$ while keeping all others fixed, define
\[
\sigma_{uv}(\{X_{ij}\}) :=
\begin{cases}
1- X_{ij}, & \text{if } (i,j)=(u,v) \text{ or } (i,j)=(v,u),\\[3pt]
X_{ij}, & \text{otherwise.}
\end{cases}
\]
Let $S_l = \sum_{k \ne l}^{n} X_{kl}$ and $S_l' = \sum_{k \ne l}^{n} \sigma_{uv}(\{X_{ij}\})_{kl}$.
It is easy to verify that
\[
|S_k - S_k'| =
\begin{cases}
1, & \text{if } k = u \text{ or } k = v,\\[3pt]
0, & \text{otherwise.}
\end{cases}
\]
By definition of $f_{n}$ in~\eqref{fn}, we obtain
\[
\begin{aligned}
\delta_{uv}(\{X_{i,j}\})
&:= f_n(\{X_{i,j}\}) - f_n(\sigma_{uv}(\{X_{ij}\})) \\
&= 2\sum_{k \ne u,v}^n
\!\left[\!
\left( u_n^{-1} - (S_u S_k)^{-1/2} \right)^2
- \left( u_n^{-1} - (S_u' S_k)^{-1/2} \right)^2
\!\right]\! X_{uk}\\
&\quad+ 2\sum_{k \ne u,v}^n
\!\left[\!
\left( u_n^{-1} - (S_v S_k)^{-1/2} \right)^2
- \left( u_n^{-1} - (S_v' S_k)^{-1/2} \right)^2
\!\right]\! X_{vk}\\
&\quad+2\left[ \left(u_n^{-1} - (S_u S_v)^{-1/2}\right)^2 X_{uv} - \left(u_n^{-1} - (S'_u S'_v)^{-1/2}\right)^2 (1 - X_{uv}) \right].
\end{aligned}
\]


Define the \textit{“typical”} event 
\[
\Gamma : = \left\{ d_{n,i} \in \left( \frac{u_n}{2}, \frac{3u_n}{2} \right) \text{ for all } 1 \leq i \leq n \right\}.
\]
We now estimate the typical Lipschitz constants $c_k$.

Consider the auxiliary function 
$g(x,y)=\left(u_n^{-1}-\frac{1}{\sqrt{xy}}\right)^2$.
A direct computation yields
\[
\frac{\partial g}{\partial x}
= 2\!\left(u_n^{-1}-\frac{1}{\sqrt{xy}}\right)
\!\left(-\frac{1}{2}x^{-3/2}y^{-1/2}\right)
= \left(u_n^{-1}-\frac{1}{\sqrt{xy}}\right)x^{-3/2}y^{-1/2}.
\]
For $x,y\in (u_n/2,3u_n/2)$, we have 
$x^{-3/2}y^{-1/2}\le 4u_n^{-2}$ 
and $\left|u_n^{-1}-x^{-1/2}y^{-1/2}\right|\le u_n^{-1}$,
hence 
$|\partial g / \partial x| \le 4u_n^{-3}$.
By symmetry, the same holds for $\partial g / \partial y$.  
Consequently, if both $x$ and $y$ change by at most $1$, 
the function $g(x,y)$ changes by at most $c u_n^{-3}$.

When the event $\Gamma$ holds, changing one coordinate $(u,v)$ in $\{X_{ij}\}$ affects at most the degrees of vertices $u$ and $v$.  
Recalling that $d_{n,i} = \sum_{j \ne i} a_{ij}$, we have
\[
\begin{aligned}
|\delta_{uv}(\mathbf{A}_{n})|
&\le 2\sum_{k \ne u,v}^n c u_n^{-3} a_{uk}
  + 2\sum_{k \ne u,v}^n c u_n^{-3} a_{vk}
  + c u_n^{-3}\\
&\le c u_n^{-3}(d_{n,u} + d_{n,v} + 1)
\le 3c u_n^{-2},
\end{aligned}
\]
since $d_{n,i}\in [u_n/2, 3u_n/2]$ under $\Gamma$.  
Hence, for $F := n^{-1} f_n$, the \textit{typical} Lipschitz constants satisfy
\[
c_{uv} \le \frac{C_1}{n\,u_n^2},
\]
whenever the event $\Gamma$  holds.

Outside the typical event $\Gamma$, we distinguish two cases.

First, if $S_u = 0$ or $S_u' = 0$, then necessarily
$X_{uk} = 0$ for all $k \neq u,v$.
In this case the corresponding contribution vanishes identically, namely,
\[
\left|
\sum_{k \ne u,v}^n
\Big[
\big( u_n^{-1} - (S_u S_k)^{-1/2} \big)^2
- \big( u_n^{-1} - (S_u' S_k)^{-1/2} \big)^2
\Big] X_{uk}
\right|
= 0 .
\]

Otherwise, assume that $S_u S_u' \neq 0$.
Using the identities
$x^2 - y^2 = (x+y)(x-y)$ and
\[
x^{-1/2} - y^{-1/2}
= \frac{x-y}{xy\big(x^{-1/2}+y^{-1/2}\big)},
\]
we obtain
\[
\begin{aligned}
&\left|
\sum_{k \ne u,v}^n
\Big[
\big( u_n^{-1} - (S_u S_k)^{-1/2} \big)^2
- \big( u_n^{-1} - (S_u' S_k)^{-1/2} \big)^2
\Big] X_{uk}
\right| \\
\le\;&
\left|
\sum_{k \ne u,v}^n
\Big(
2u_n^{-1}
- (S_u S_k)^{-1/2}
- (S_u' S_k)^{-1/2}
\Big)
\Big(
(S_u S_k)^{-1/2}
- (S_u' S_k)^{-1/2}
\Big)
X_{uk}
\right| \\
\le\;&
c \left|
\sum_{k \ne u,v}^n
\Big(
(S_u S_k)^{-1/2}
- (S_u' S_k)^{-1/2}
\Big)
X_{uk}
\right| \\
\le\;&
c \, \max_k S_k^{-1/2}
\frac{|S_u - S_u'|}{S_u S_u'\big(S_u^{-1/2}+S_u'^{-1/2}\big)}
\sum_{k \ne u,v}^n X_{uk} \\
\le\;&
c \, \max_k S_k^{-1/2}
\frac{S_u}{S_u S_u'\big(S_u^{-1/2}+S_u'^{-1/2}\big)} \\
=\;&
c \, \max_k S_k^{-1/2}
\frac{1}{S_u'\big(S_u^{-1/2}+S_u'^{-1/2}\big)}.
\end{aligned}
\]
Under the pseudo-inverse convention, if there exists  $1 \le i\le n$ such that $d_{n,i}=O(1)$,
then the increment $\delta_{uv}$ may attain its maximal size, and we obtain the crude bound
\[
\delta_{uv} = O(1).
\]

Therefore, recall that $F := n^{-1} f_n$, we also have the crude \textit{worst-case} bound
\[
c'_{uv} \le \frac{C_2}{n},
\]
for some absolute constant $C_2 > 0$.

We choose the parameters uniformly over edges $k \leftrightarrow (u,v)$ by
\[
\beta_k = \beta_n := u_n^{-2} \in (0,1] \quad \text{for all sufficiently large } n,
\]
so that $\varepsilon_k = \beta_n (c_k' - c_k) \le C_2 n^{-1} u_n^{-2}$.  
Since $u_n \to \infty$ under the assumption $n p_n \gtrsim \log n$, the typical constants dominate:
\[
c_k + \varepsilon_k \le (C_1+C_2) n^{-1} u_n^{-2}.
\]
Summing over the $N = \binom{n}{2}$ independent edges gives
\[
\sum_k (c_k + \varepsilon_k)^2 
\le N \cdot \frac{(C_1+C_2)^2}{n^2 u_n^4} 
= O(u_n^{-4}).
\]
Therefore, by~\eqref{tbd}, we obtain the sub-Gaussian tail bound
\begin{equation}\label{new2}
   \mathbb{P}(F - \mathbb{E}F \ge t \text{ and } \neg \mathcal{A}) 
   \le \exp(- c\, t^2 u_n^4).
\end{equation}
for some absolute constant $c > 0$.

It remains to control $\mathbb{P}(\mathcal{A})$.  
Since $d_{n,i} \sim \mathrm{Bin}(n-1, p_n)$ and $u_n = (n-1)p_n$, Chernoff’s bound gives
\[
\mathbb{P}(|d_{n,i} - u_n| \ge u_n / 2) \le 2 e^{-u_n / 12}.
\]

By a union bound,
\[
\mathbb{P}(\neg \Gamma) \le 2n e^{-u_n / 12}.
\]
 This also implies $\mathbf{D}_n$ is invertible asymptotically almost surely.  Using $\beta_n = u_n^{-2}$, we deduce from~\eqref{tbd} that
\begin{equation}\label{new3}
    \mathbb{P}(\mathcal{A}) 
    \le N \beta_n^{-1} \mathbb{P}(\neg \Gamma) 
    \lesssim n^3 u_n^2 e^{-u_n / 12},
\end{equation}
The right-hand side in \eqref{new3} is summable in \(n\) under condition \eqref{lem3-cond}, which states that \(u_n=(n-1)p_n\gtrsim \log n\).

Finally, set $t_n = k / u_n$ with any fixed $k > 0$.  
By~\eqref{new2} and~\eqref{new3},
\[
\sum_{n=1}^\infty 
\mathbb{P}(F - \mathbb{E}F \ge t_n) 
\le \sum_{n=1}^\infty 
\mathbb{P}(F - \mathbb{E}F \ge t_n \text{ and } \neg \mathcal{A})
+ \sum_{n=1}^\infty \mathbb{P}(\mathcal{A})
< \infty.
\]

By the Borel–Cantelli lemma,
\[
\frac{1}{n}\operatorname{tr}(\widetilde{\mathbf{\mathcal L}}_n - \mathbf{\mathcal L}_n)^2
- \frac{1}{n}\mathbb{E}\!\left[\operatorname{tr}(\widetilde{\mathbf{\mathcal L}}_n - \mathbf{\mathcal L}_n)^2\right]
= o\!\left(\frac{1}{u_n}\right)
\qquad \text{a.s. as } n \to \infty,
\]
which establishes the desired concentration result.
\end{proof}

\subsection{Proof of Theorem~\ref{thm1}}

\begin{proof}
It suffices to show that 
\[
d_{\mathrm{BL}}\!\left(\hat{\mu}(\mathbf{\mathcal L}_n), \hat{\mu}(\widetilde{\mathbf{\mathcal L}}_n)\right)
\longrightarrow 0
\quad \text{a.s. as } n \to \infty.
\]

By Lemmas~\ref{lem1} and~\ref{lem2}, we have
\begin{align*}
d_{\mathrm{BL}}^2\!\left(\hat{\mu}(\widetilde{\mathbf{\mathcal L}}_n),\hat{\mu}(\mathbf{\mathcal L}_n)\right)
&\le \frac{1}{n}\operatorname{tr}\!\left(\widetilde{\mathbf{\mathcal L}}_n - \mathbf{\mathcal L}_n\right)^2 \\
&\le \Biggl|
\frac{1}{n}\operatorname{tr}\!\left(\widetilde{\mathbf{\mathcal L}}_n - \mathbf{\mathcal L}_n\right)^2
- \frac{1}{n}\mathbb{E}\!\left[\operatorname{tr}\!\left(\widetilde{\mathbf{\mathcal L}}_n - \mathbf{\mathcal L}_n\right)^2\right]
\Biggr| + \frac{1}{n}\mathbb{E}\!\left[\operatorname{tr}\!\left(\widetilde{\mathbf{\mathcal L}}_n - \mathbf{\mathcal L}_n\right)^2\right].
\end{align*}

Lemma~\ref{lem1} yields
\[
\frac{1}{n}\mathbb{E}\!\left[\operatorname{tr}\!\left(\widetilde{\mathbf{\mathcal L}}_n - \mathbf{\mathcal L}_n\right)^2\right]
= O(u_n^{-2}),
\]
while Lemma~\ref{lem2} implies that the centered fluctuation term is
$o(u_n^{-1})$ almost surely. Consequently,
\begin{equation}\label{asymptotically identical spectra}
    d_{\mathrm{BL}}^2\!\left(\hat{\mu}(\widetilde{\mathbf{\mathcal L}}_n),\hat{\mu}(\mathbf{\mathcal L}_n)\right)
= o(u_n^{-1})
\quad \text{a.s. as } n \to \infty.
\end{equation}

\eqref{asymptotically identical spectra} is also true under a linear scaling of $\{\gamma_n\}$ and $\{\rho_n\}$,
we have
\begin{equation}\label{eq:scaled_est}
d_{\mathrm{BL}}^2\!\left(
\hat{\mu}(\gamma_n \mathbf{\mathcal L}_n + \rho_n \mathbf{I}_n), 
\hat{\mu}(\gamma_n \widetilde{\mathbf{\mathcal L}}_n + \rho_n \mathbf{I}_n)
\right)
= o\!\left(\frac{\gamma_n^2}{u_n}\right)
\quad \text{a.s. as } n \to \infty.
\end{equation}

On the other hand, recall \eqref{wc2}, we know that
\[
\hat{\mu}(\gamma_n \widetilde{\mathbf{\mathcal L}}_n + \rho_n \mathbf{I}_n)
\Longrightarrow \mu_{\mathrm{sc}}
\quad \text{a.s. as } n \to \infty,
\]
which is equivalent to
\[
d_{\mathrm{BL}}\!\left(
\hat{\mu}(\gamma_n \widetilde{\mathbf{\mathcal L}}_n + \rho_n \mathbf{I}_n),
\mu_{\mathrm{sc}}
\right)
\longrightarrow 0
\quad \text{a.s. as } n \to \infty.
\]

Note that 
\[
\gamma_n = -\rho_n = -\sqrt{\frac{n p_n}{1-p_n}}.
\]

Under the assumption $n p_n \ge C \log n$, we have $u_n \asymp n p_n$ and hence
\[
\frac{\gamma_n^2}{u_n}= O(1).
\]

Therefore, by the triangle inequality for $d_{\mathrm{BL}}$,
\[
\begin{aligned}
d_{\mathrm{BL}}\!\left(
\hat{\mu}(\gamma_n \mathbf{\mathcal L}_n + \rho_n \mathbf{I}_n),
\mu_{\mathrm{sc}}
\right)
&\le d_{\mathrm{BL}}\!\left(
\hat{\mu}(\gamma_n \mathbf{\mathcal L}_n + \rho_n \mathbf{I}_n),
\hat{\mu}(\gamma_n \widetilde{\mathbf{\mathcal L}}_n + \rho_n \mathbf{I}_n)
\right) \\
&\quad + d_{\mathrm{BL}}\!\left(
\hat{\mu}(\gamma_n \widetilde{\mathbf{\mathcal L}}_n + \rho_n \mathbf{I}_n),
\mu_{\mathrm{sc}}
\right) \longrightarrow 0
\quad \text{a.s. as } n \to \infty. \\
&
\end{aligned}
\]
This completes the proof of Theorem~\ref{thm1}.  
\end{proof}

\section{Proof of Theorem~\ref{thm:CLV}}\label{sec:proof_CLV}

\begin{proof}
Define the normalized adjacency matrices
\[
\mathbf{N}_n := \mathbf{D}_n^{-1/2} \mathbf{A}_n \mathbf{D}_n^{-1/2},
\qquad
\mathbf{\widetilde{N}}_n := \mathbf{W}_n^{-1/2} \mathbf{A}_n \mathbf{W}_n^{-1/2}.
\]

We proceed in two steps. 
First, we establish the semicircle law for the matrix $\mathbf{\widetilde N}_n$.
Second, we show that the same limiting law holds for $N_n$ by a perturbation analysis.

Recall that Chung, Lu, and Vu~\cite{chung2003spectra} proved that, under condition~\eqref{clvcondition}, the normalized matrix
\[
\mathbf{C}_{\mathrm{nor}} := \sqrt{\bar w_n}\, \mathbf{C},
\qquad
\mathbf{C} := \mathbf{W}_n^{-1/2} \mathbf{A}_n \mathbf{W}_n^{-1/2} - \phi_n W_n^{1/2} \mathbf{K W}_n^{1/2},
\]
has an empirical spectral distribution converging in probability to the semicircle distribution $\mu_{sc}$.
We note that in \cite{chung2003spectra}, the semicircle law was established only in expectation.  However, the convergence in probability follows from a standard
second-moment argument: for each fixed $s\geq 1$,
\[
    \operatorname{Var}\!\left(
        \frac{1}{n}\operatorname{Tr}(C_{\mathrm{nor}}^s)
    \right)
    = O_s(n^{-1}),
\]
so that Chebyshev's inequality upgrades convergence of the expected
moments to convergence in probability.

Observe that
\[
\mathbf{\widetilde N_n}
= \mathbf{W}_n^{-1/2} \mathbf{A}_n \mathbf{W}_n^{-1/2}
= \mathbf{C} + \phi_n \mathbf{W}_n^{1/2} \mathbf{K W}_n^{1/2}
=: \mathbf{C} + \mathbf{R}_W ,
\]
where $\mathbf{K}$ denotes the all-ones matrix. 
The matrix $\mathbf{R}_W$ has rank one, since
\[
\mathbf{W}_n^{1/2} \mathbf{K W}_n^{1/2}
=
\bigl( \sqrt{w_{n,1}}, \dots, \sqrt{w_{n,n}} \bigr)^{\!\top}
\bigl( \sqrt{w_{n,1}}, \dots, \sqrt{w_{n,n}} \bigr).
\]
Consequently, $\mathbf{\widetilde N}_n$ differs from $\mathbf{C}$ by a rank-one perturbation.

We now invoke the rank inequality due to Bai~\cite[Lemma~2.6]{bai2010spectral}.

\begin{lemma}[Rank inequality]\label{rankinequality}
Let $\mathbf{A}$ and $\mathbf{B}$ be two $n\times n$ Hermitian matrices. Then
\[
\bigl\| \hat{\mu}(\mathbf{A}) - \hat{\mu}(\mathbf{B}) \bigr\|_{\infty}
\le \frac{1}{n}\,\operatorname{rank}(\mathbf{A- B}).
\]
\end{lemma}

Applying Lemma~\ref{rankinequality} yields
\begin{equation}\label{rankzero}
\bigl\| \hat{\mu}(\mathbf{\widetilde N}_n) - \hat{\mu}(\mathbf{C}) \bigr\|_{\infty}
\le \frac{1}{n},
\end{equation}
which implies that $\mathbf{\widetilde N}_n$ and $\mathbf{C}$ have the same limiting spectral distribution.

Combining \eqref{rankzero} observation with the semicircle law for $\mathbf{C}_{\mathrm{nor}}$, we conclude that
\begin{equation}\label{clv}
\hat{\mu}\!\left( \sqrt{\bar w_n}\, \mathbf{\widetilde N}_n \right)
\;\Longrightarrow\;
\mu_{\mathrm{sc}}
\quad \text{in probability as } n \to \infty.
\end{equation}

\begin{lemma}\label{key}
There exist absolute constants $c,c'>0$ such that
\[
\frac{1}{n}\,\mathbb{E}\,\operatorname{tr}\!\left( (N_n-\widetilde N_n)^2 \right)
\;\le\;
\frac{c}{\bar w_n\, w_{n,\min}}
\;+\;
c\,\bar w_n\, e^{-c' w_{n,\min}} .
\]
In particular, since $w_{n,\min}\to\infty$, we have
\[
\begin{aligned}
\bar w_n\cdot\frac1n
\mathbb E\operatorname{tr}(\mathbf{N}_n-\widetilde {\mathbf{N}}_n)^2
&\le
\frac{c}{w_{n,\min}}
+c\bar w_n^2e^{-c'w_{n,\min}}  \\
&=O\left(\frac1{w_{n,\min}}\right)
\longrightarrow0
\qquad\text{as }n\to\infty.
\end{aligned}
\]
\end{lemma}

\begin{proof}
For $i\neq j$, the entries of $\mathbf{N}_n$ and $\mathbf{\widetilde N}_n$ are given by
\[
(\mathbf{N}_n)_{ij} = \frac{a_{ij}^{(n)}}{\sqrt{d_{n,i} d_{n,j}}},
\qquad
(\mathbf{\widetilde N}_n)_{ij} = \frac{a_{ij}^{(n)}}{\sqrt{w_{n,i} w_{n,j}}}.
\]

Hence,
\[
\mathbf{\widetilde N}_n - \mathbf{N}_n
=
\Bigl(
\Bigl(\frac{1}{\sqrt{w_{n,i} w_{n,j}}}
      - \frac{1}{\sqrt{d_{n,i} d_{n,j}}}\Bigr)
a_{ij}^{(n)}
\Bigr)_{1\le i,j\le n}.
\]

By symmetry and independence of the edge indicators,
\begin{equation}\label{etrnn}
\begin{aligned}
\frac{1}{n}\mathbb{E}\operatorname{tr}\!\left( (\mathbf{\widetilde N}_n - \mathbf{N}_n)^2 \right)
&= \frac{1}{n}
\sum_{i\neq j}
\mathbb{E}\!\left[
\Bigl(\frac{1}{\sqrt{w_{n,i} w_{n,j}}}
      - \frac{1}{\sqrt{d_{n,i} d_{n,j}}}\Bigr)^2
a_{ij}^{(n)}
\right].
\end{aligned}
\end{equation}

Fix $i\neq j$.
By conditioning on $a_{ij}^{(n)}$,
\[
\mathbb{E}\!\left[
\Bigl(\frac{1}{\sqrt{d_{n,i} d_{n,j}}}
      - \frac{1}{\sqrt{w_{n,i} w_{n,j}}}\Bigr)^2
a_{ij}^{(n)}
\right]
=
p_{n,ij}\,
\mathbb{E}\!\left[
\Bigl(\frac{1}{\sqrt{d_{n,i} d_{n,j}}}
      - \frac{1}{\sqrt{w_{n,i} w_{n,j}}}\Bigr)^2
\,\Big|\,a_{ij}^{(n)}=1
\right].
\]

Define the event
\[
E_3
:=
\Bigl\{
d_{n,i}\in[\tfrac12 w_{n,i},\,\tfrac32 w_{n,i}],
\;
d_{n,j}\in[\tfrac12 w_{n,j},\,\tfrac32 w_{n,j}]
\Bigr\}.
\]

On $E_3$, a Taylor expansion argument yields
\[
\mathbb{E}\!\left[
\Bigl(\frac{1}{\sqrt{d_{n,i} d_{n,j}}}
      - \frac{1}{\sqrt{w_{n,i} w_{n,j}}}\Bigr)^2
\mathbf 1_{E_3}
\,\Big|\,a_{ij}^{(n)}=1
\right]
\le
c\!\left(
\frac{1}{w_{n,i}^2 w_{n,j}}
+
\frac{1}{w_{n,j}^2 w_{n,i}}
\right).
\]

Since $d_{n,i}$ and $d_{n,j}$ are sums of independent Bernoulli random variables,
by a Chernoff bound, we have
\[
\mathbb{P}(E_3^c)
\le
2e^{-w_{n,i}/12}+2e^{-w_{n,j}/12}
\le
4e^{-w_{n,\min}/12}.
\]

On $E_3^c$, we use the crude bound and the pseudo-inverse convention to have
\[
\Bigl|
\frac{1}{\sqrt{w_{n,i} w_{n,j}}}
-
\frac{1}{\sqrt{d_{n,i} d_{n,j}}}
\Bigr|
\le
\frac{1}{\sqrt{w_{n,i} w_{n,j}}}
+
\frac{1}{\sqrt{d_{n,i} d_{n,j}}}
\le 2,
\]
which implies
\[
\mathbb{E}\!\left[
\Bigl(\frac{1}{\sqrt{d_{n,i} d_{n,j}}}
      - \frac{1}{\sqrt{w_{n,i} w_{n,j}}}\Bigr)^2
\mathbf 1_{E_3^c}
\,\Big|\,a_{ij}^{(n)}=1
\right]
\le
16 e^{-w_{n,\min}/12}.
\]

Combining the two cases, we obtain
\[
\mathbb{E}\!\left[
\Bigl(\frac{1}{\sqrt{d_{n,i} d_{n,j}}}
      - \frac{1}{\sqrt{w_{n,i} w_{n,j}}}\Bigr)^2
\,\Big|\,a_{ij}^{(n)}=1
\right]
\le
c\!\left(
\frac{1}{w_{n,i}^2 w_{n,j}}
+
\frac{1}{w_{n,j}^2 w_{n,i}}
\right)
+
16 e^{-w_{n,\min}/12}.
\]

Multiplying by $p_{n,ij}=\phi_n w_{n,i}w_{n,j}$ gives
\[
\mathbb{E}\!\left[
\Bigl(\frac{1}{\sqrt{d_{n,i} d_{n,j}}}
      - \frac{1}{\sqrt{w_{n,i} w_{n,j}}}\Bigr)^2
a_{ij}^{(n)}
\right]
\le
c\phi_n\!\left(
\frac{1}{w_{n,i}}+\frac{1}{w_{n,j}}
\right)
+
16p_{n,ij}e^{-w_{n,\min}/12}.
\]

Summing over $i\neq j$ and using
$\sum_{i,j}p_{n,ij}=n\bar w_n=1/\phi_n$, we conclude from
\eqref{etrnn} that
\[
\frac{1}{n}\mathbb{E}\operatorname{tr}\!\left( (\mathbf{\widetilde N}_n - \mathbf{N}_n)^2 \right)
\le
\frac{c}{\bar w_n w_{n,\min}}
+
c\bar w_n e^{-w_{n,\min}/12},
\]
which completes the proof.
\end{proof}

Recall that, for any sequences $\gamma_n \in \mathbb{R}$, we have
\[
d_{BL}^2 \left( \hat{\mu}(\gamma_n \mathbf{N}_n), \hat{\mu}(\gamma_n \mathbf{\widetilde{N}}_n) \right) \leq \gamma_n^2 \cdot \frac{1}{n} \operatorname{tr} \left( (\mathbf{N}_n - \mathbf{\widetilde{N}}_n)^2 \right),
\]

Taking $\gamma_n = \sqrt{\bar{w}_n}$, and using the expectation value and Lemma \ref{key}, we get:
\[
\mathbb{E} \left[ d_{BL}^2 \left( \hat{\mu}(\sqrt{\bar{w}_n} \mathbf{N}_n), \hat{\mu}(\sqrt{\bar{w}_n} \mathbf{\widetilde{N}}_n) \right) \right] \leq \bar{w}_n \cdot \frac{1}{n} \mathbb{E} \operatorname{tr} \left( (\mathbf{N}_n - \mathbf{\widetilde{N}}_n)^2 \right) = O \left( \frac{1}{w_{n,\min}} \right)  \to 0 \quad \text{as n}\to \infty,
\]
where we use the assumption $w_{n,\min}\gg \sqrt{\overline{w}_n}$. Thus, by Markov's inequality, we obtain
\[
d_{BL} \left( \hat{\mu}(\sqrt{\bar{w}_n} \mathbf{N}_n), \hat{\mu}(\sqrt{\bar{w}_n} \mathbf{\widetilde{N}}_n) \right) \to 0\quad \text{in probability as } n \to \infty.
\]

Thus, by (\ref{clv}) and the triangle inequality,
\begin{equation}\label{last}
\begin{aligned}
d_{BL} \left( \hat{\mu}(\sqrt{\bar{w}_n} \mathbf{N}_n), \mu_{\text{sc}} \right)
&\le d_{BL} \left( \hat{\mu}(\sqrt{\bar{w}_n} \mathbf{N}_n), \hat{\mu}(\sqrt{\bar{w}_n} \mathbf{\widetilde{N}}_n) \right) \\
&\quad +d_{BL}\left( \hat{\mu}(\sqrt{\bar{w}_n} \mathbf{\widetilde{N}}_n), \mu_{\text{sc}} \right) \\
&\longrightarrow 0
\quad \text{in probability as } n \to \infty.
\end{aligned}
\end{equation}
And \eqref{last} proves
\[
\hat{\mu}(\sqrt{\bar{w}_n} \mathbf{D}_n^{-1/2}\mathbf{A}_n\mathbf{D}_n^{-1/2}) \Rightarrow \mu_{\text{sc}} \quad \text{in probability as } n \to \infty
\]
as desired.

\end{proof}

\subsection*{Acknowledgments} Y.Z. was partially supported by the Simons Grant MPS-TSM-00013944 and NSF DMS-2606337. Y.C. was supported by the National Natural Science Foundation of China
(12595294, 12231002).

\bibliographystyle{imsart-number}

\bibliography{ref}

\end{document}